\documentclass[10pt]{article}

\usepackage{amssymb,amsmath,amsfonts,amsthm}
\usepackage[a4paper,includeheadfoot,left=3.7cm,right=3.7cm,top=2.9cm,bottom=3cm]{geometry}
\usepackage{mathtools}

\usepackage[numbers,comma,sort&compress]{natbib}
\usepackage[stretch=10]{microtype} 

\makeatletter
\let\@fnsymbol\@arabic
\makeatother

\usepackage{color}
\definecolor{marin}{rgb}{0.,0.3,0.7}
\usepackage[colorlinks,citecolor=marin,linkcolor=marin,urlcolor=marin,bookmarksopen,bookmarksnumbered,breaklinks=true]{hyperref}

\providecommand{\abs}[1]{\lvert#1\rvert}
\providecommand{\absbig}[1]{\bigl\lvert#1\bigr\rvert}
\providecommand{\kla}[1]{(#1)}
\providecommand{\klabig}[1]{\bigl(#1\bigr)}
\providecommand{\klaBig}[1]{\Bigl(#1\Bigr)}
\providecommand{\norm}[1]{\lVert#1\rVert}
\providecommand{\normbig}[1]{\bigl\lVert#1\bigr\rVert}
\DeclareMathOperator{\diag}{diag}
\DeclareMathOperator{\sinc}{sinc}
\DeclareMathOperator{\real}{Re}
\newcommand{\iu}{\mathrm{i}}
\newcommand{\e}{\mathrm{e}}
\newtheorem{theorem}{Theorem}[section]
\newtheorem{lemma}[theorem]{Lemma}
\theoremstyle{definition}
\newtheorem{remark}[theorem]{Remark}

\title{On energy conservation by trigonometric integrators in the linear case with application to wave equations}

\author{Ludwig Gauckler\,\thanks{Institut f\"ur Mathematik,
          Freie Universit\"at Berlin,
          Arnimallee 9,
          D-14195 Berlin, Germany.}
}

\date{Version of 11 May 2018}

\begin{document}

\maketitle

\begin{abstract}
Trigonometric integrators for oscillatory linear Hamiltonian differential equations are considered. 
Under a condition of Hairer \& Lubich on the filter functions in the method, a modified energy is derived that is exactly preserved by trigonometric integrators. This implies and extends a known result on all-time near-conservation of energy. The extension can be applied to linear wave equations.\\[1.5ex]
\end{abstract}

\section{Introduction}

Trigonometric integrators form a popular class of numerical methods for oscillatory second-order differential equations; see \cite[Chapter XIII]{Hairer2006}. The various available trigonometric integrators differ (only) by the filter functions that are used inside the methods. 
One way to choose filter functions was put forward by Hairer \& Lubich in~\cite{Hairer2000}, which led to the well-known trigonometric integrators of Hairer \& Lubich~\cite{Hairer2000} and Grimm \& Hochbruck~\cite{Grimm2006}. It is a condition on the filter functions (see~\eqref{eq-hl} below) that can be proven to imply very good energy conservation by the corresponding trigonometric integrators on long time intervals: In the case of \emph{linear} oscillatory Hamiltonian differential equations, all-time near-conservation of energy can be proven for \emph{all} step-sizes; see~\cite{Hairer2000}. In other words, there are no numerical resonances (in the linear case) that show up in other trigonometric integrators on long time intervals. 

In the present note, we consider trigonometric integrators under the mentioned condition on the filter functions of Hairer \& Lubich~\cite{Hairer2000} in the mentioned situation of linear oscillatory Hamiltonian differential equations. We show that there exists a \emph{modified energy} that is \emph{exactly} preserved by the numerical method. 
This modified energy is close to the original energy, which yields a new proof and an extension of the mentioned result of~\cite{Hairer2000} on all-time near-conservation of energy for all step-sizes.

This extension can be applied to linear wave equations. We use it to prove all-time near-conservation of energy for a spectral semi-discretization of linear wave equations, again without any restriction on the time step-size, neither of CFL-type nor of resonance-excluding nature. This seems to be the first long-time result for temporal discretizations of Hamiltonian partial differential equations that is completely uniform in the time step-size.

\section{Oscillatory Hamiltonian systems}

We consider oscillatory Hamiltonian systems of the form
\begin{equation}\label{eq-ode}
\ddot{q} = -\Omega^2 q + g(q), \qquad q=q(t)\in\mathbb{C}^d.
\end{equation}
In this equation, the matrix
\[
  \Omega=\diag(\omega_j)_{j=1}^d\in\mathbb{R}^{d\times d}
\]
is a diagonal matrix containing nonnegative, possibly large frequencies $\omega_j\in\mathbb{R}$. We denote by~$\omega$ the smallest nonzero frequency:
\[
\omega = \min_{j: \, \omega_j>0} \omega_j.
\]
The term
\[
g(q)=-\nabla U(q)
\]
in~\eqref{eq-ode} stems from a (sufficiently regular, in particular real differentiable) potential $U\colon\mathbb{C}^d\rightarrow \mathbb{R}$. The complex gradient $\nabla$ with respect to $q\in\mathbb{C}^d$ is defined as $\nabla=\nabla_x+\iu\nabla_y$ with the real part $x\in\mathbb{R}^d$ and the imaginary part $y\in\mathbb{R}^d$ of $q=x+\iu y$.\footnote{We remark that the choice of a complex setting is with a view towards the application to wave equations in the final part of the paper.}
We will be interested in the case that~\eqref{eq-ode} is a linear equation, i.e., 
\begin{equation}\label{eq-A}
  g(q) = - A q, \qquad 
U(q) = \tfrac12 q^* A q, \qquad A\in\mathbb{C}^{d\times d}\text{ self-adjoint,}
\end{equation}
where~$*$ denotes the conjugate transpose.

The total energy of the Hamiltonian system~\eqref{eq-ode} is given by
\begin{equation}\label{eq-energy}
H(q,\dot{q}) = \tfrac12 \norm{\Omega q}^2 + \tfrac12 \norm{\dot{q}}^2 + U(q),
\end{equation}
where $\norm{\cdot}$ denotes the Euclidean norm on~$\mathbb{C}^d$.

\section{Trigonometric integrators}

Trigonometric integrators form a popular class of numerical methods for oscillatory Hamiltonian systems~\eqref{eq-ode}. We consider here symmetric trigonometric integrators with the step-size~$h$ (see \cite[Chapter XIII]{Hairer2006}):
\begin{subequations}\label{eq-trigo}\begin{align}
q_{n+1} &= \cos(h\Omega) q_n + h\sinc(h\Omega) \dot{q}_n + \tfrac12 h^2 \sinc(h\Omega) \Psi_1 g(\Phi q_n),\label{eq-trigo-1}\\
\dot{q}_{n+1} &= -\Omega \sin(h\Omega) q_n + \cos(h\Omega) \dot{q}_n + \tfrac12 h \klabig{ \cos(h\Omega) \Psi_1 g(\Phi q_n) + \Psi_1 g(\Phi q_{n+1})}.\label{eq-trigo-2}
\end{align}\end{subequations}
Throughout the paper, we assume $h\le 1$.
In the method~\eqref{eq-trigo}, we use filter operators  
\[ 
\Psi_1=\psi_1(h\Omega) \qquad\text{and}\qquad \Phi=\phi(h\Omega)
\]
that are computed from real-valued and even filter functions~$\psi_1$ and~$\phi$. We will use the natural and usual conditions 
\begin{equation}\label{eq-filter-bounds}
\abs{\psi_1(\xi)} \le c_0, \qquad \abs{\phi(\xi)}\le c_0, \qquad \abs{\phi(\xi)-1}\le c_1 \abs{\xi} \qquad\text{for all }\, \xi\in\mathbb{R}
\end{equation}
on the filter functions. 
The condition of Hairer \& Lubich~\cite{Hairer2000} as mentioned in the introduction is
\begin{equation}\label{eq-hl}
\psi_1 = \sinc\cdot \,\phi.
\end{equation}

A symmetric trigonometric integrator~\eqref{eq-trigo} can be interpreted as a splitting integrator applied to an averaged version of~\eqref{eq-ode}. In fact, the integrator~\eqref{eq-trigo} can be written as
\begin{subequations}\label{eq-trigo-split}\begin{align}
\dot{q}_{n,+} &= \dot{q}_n + \tfrac12 h \Psi_1 g(\Phi q_n),\label{eq-trigo-split-1}\\
\begin{pmatrix} q_{n+1}\\ \dot{q}_{n+1,-}\end{pmatrix} &= 
\begin{pmatrix} \cos(h\Omega) & h\sinc(h\Omega)\\ -\Omega \sin(h\Omega) & \cos(h\Omega) \end{pmatrix} 
\begin{pmatrix} q_{n}\\ \dot{q}_{n,+}\end{pmatrix},\label{eq-trigo-split-2}\\
\dot{q}_{n+1} &= \dot{q}_{n+1,-} + \tfrac12 h \Psi_1 g(\Phi q_{n+1}),\label{eq-trigo-split-3}
\end{align}\end{subequations}
which is a Strang splitting: The second line~\eqref{eq-trigo-split-2} is the solution of $\mathrm{d} q/\mathrm{d} t = \dot{q}$, $\mathrm{d}\dot{q}/\mathrm{d} t = -\Omega^2 q$ over a time interval~$h$ for initial values $q_n$, $\dot{q}_{n,+}$. The first line~\eqref{eq-trigo-split-1} and the third line~\eqref{eq-trigo-split-3} are solutions of $\mathrm{d} q/\mathrm{d} t = 0$, $\mathrm{d} \dot{q}/\mathrm{d} t = \Psi_1 g(\Phi q)$ over a time interval $h/2$ for initial values $q_n$, $\dot{q}_n$ and $q_{n+1}$, $\dot{q}_{n+1,-}$, respectively.

\section{A modified energy}

In this section, we prove our main result which states that, under the condition~\eqref{eq-hl} and in the linear case, the trigonometric integrator~\eqref{eq-trigo} exactly conserves a modified energy. The derivation of this modified energy is based on~\cite{Gaucklera}.

\begin{theorem}\label{thm-ec}
Consider the oscillatory system~\eqref{eq-ode} in the linear case~\eqref{eq-A} discretized with a trigonometric integrator~\eqref{eq-trigo} whose filter functions satisfy~\eqref{eq-hl}. Then, we have
\[
\mathcal{H}(q_{n+1},\dot{q}_{n+1}) = \mathcal{H}(q_{n},\dot{q}_{n})
\]
with the modified energy
\begin{equation}\label{eq-Ec}
\mathcal{H}(q,\dot{q}) = \tfrac12 \norm{\Omega q}^2 + \tfrac12 \norm{\dot{q}}^2 + \tfrac12 \real\klabig{\kla{\cos(h\Omega) \Phi q}^* A \Phi q} - \tfrac18 h^2 \norm{\Psi_1 A \Phi q}^2.
\end{equation}
\end{theorem}
\begin{proof}
We consider the trigonometric integrator~\eqref{eq-trigo} written as a splitting method~\eqref{eq-trigo-split}. In the second step~\eqref{eq-trigo-split-2}, we have
\[
\norm{\Omega q_{n+1}}^2 + \norm{\dot{q}_{n+1,-}}^2 = \norm{\Omega q_{n}}^2 + \norm{\dot{q}_{n,+}}^2,
\]
since the matrix $\begin{psmallmatrix}\cos(h\Omega) & \sin(h\Omega)\\ -\sin(h\Omega) & \cos(h\Omega)\end{psmallmatrix}$ is unitary.
Replacing $\dot{q}_{n+1,-}$ and $\dot{q}_{n,+}$ in this equation with~\eqref{eq-trigo-split-1} and~\eqref{eq-trigo-split-3} yields
\begin{equation}\label{eq-aux1}\begin{split}
\norm{\Omega q_{n+1}}^2 &+ \norm{\dot{q}_{n+1}}^2 - \real\klabig{h \dot{q}_{n+1}^* \Psi_1 g(\Phi q_{n+1})} + \tfrac14 h^2 \norm{\Psi_1 g(\Phi q_{n+1})}^2\\
 &= \norm{\Omega q_{n}}^2 + \norm{\dot{q}_{n}}^2 + \real\klabig{h \dot{q}_{n}^* \Psi_1 g(\Phi q_{n})} + \tfrac14 h^2 \norm{\Psi_1 g(\Phi q_{n})}^2.
\end{split}\end{equation}

We now rewrite the mixed terms $- h \dot{q}_{n+1}^* \Psi_1 g(\Phi q_{n+1})$ and $h \dot{q}_{n}^* \Psi_1 g(\Phi q_{n})$, which have opposite signs. For the term $h \dot{q}_{n}^* \Psi_1 g(\Phi q_{n})$, we have
\begin{equation}\label{eq-mixedterm-1}\begin{split}
    h \dot{q}_{n}^* \Psi_1 g(\Phi q_{n}) &= \klabig{h \sinc(h\Omega)\dot{q}_{n}}^* \Phi g(\Phi q_{n})\\
    &= \klabig{q_{n+1} - \cos(h\Omega) q_n - \tfrac12 h^2 \sinc(h\Omega) \Psi_1 g(\Phi q_n)}^* \Phi g(\Phi q_{n})\\
 &= \kla{\Phi q_{n+1}}^* g(\Phi q_{n}) - \kla{\cos(h\Omega) \Phi q_n}^* g(\Phi q_{n}) - \tfrac12 h^2 \norm{\Psi_1 g(\Phi q_n)}^2,
\end{split}\end{equation}
where the assumption~\eqref{eq-hl} on the filter functions is used in the first and third equality and~\eqref{eq-trigo-1} in the second equality. To rewrite the other mixed term $- h \dot{q}_{n+1}^* \Psi_1 g(\Phi q_{n+1})$ in~\eqref{eq-aux1}, we proceed as in~\eqref{eq-mixedterm-1}, but instead of~\eqref{eq-trigo-1} we use
\[
h \sinc(h\Omega) \dot{q}_{n+1} = - q_n + \cos(h\Omega) q_{n+1} + \tfrac12 h^2 \sinc(h\Omega) \Psi_1 g(\Phi q_{n+1}) ,
\]
which follows from subtracting~\eqref{eq-trigo-1} multiplied with $\cos(h\Omega)$ from~\eqref{eq-trigo-2} multiplied with $h\sinc(h\Omega)$ (or, more abstractly, from the symmetry of the method and~\eqref{eq-trigo-1}).
Together with~\eqref{eq-hl}, this yields
\begin{equation}\label{eq-mixedterm-2}\begin{split}
    - h \dot{q}_{n+1}^* \Psi_1 g(\Phi q_{n+1}) &= - \klabig{h \sinc(h\Omega)\dot{q}_{n+1}}^* \Phi g(\Phi q_{n+1})\\
    &= \klabig{q_n - \cos(h\Omega) q_{n+1} - \tfrac12 h^2 \sinc(h\Omega) \Psi_1 g(\Phi q_{n+1})}^* \Phi g(\Phi q_{n+1})\\
    &= \kla{\Phi q_{n}}^* g(\Phi q_{n+1})\\
    &\qquad\qquad- \kla{\cos(h\Omega) \Phi q_{n+1}}^* g(\Phi q_{n+1}) - \tfrac12 h^2 \norm{\Psi_1 g(\Phi q_{n+1})}^2.
\end{split}\end{equation}

Inserting the equations~\eqref{eq-mixedterm-1} and~\eqref{eq-mixedterm-2} into~\eqref{eq-aux1} and dividing by two yields
\begin{equation}\label{eq-ec-nonlinear}
\mathcal{H}(q_{n+1},\dot{q}_{n+1}) + \tfrac12 \real\klabig{\kla{\Phi q_{n}}^* g(\Phi q_{n+1})} = \mathcal{H}(q_{n},\dot{q}_{n}) + \tfrac12 \real\klabig{\kla{\Phi q_{n+1}}^* g(\Phi q_{n})}
\end{equation}
with the modified energy
\[
\mathcal{H}(q,\dot{q}) = \tfrac12 \norm{\Omega q}^2 + \tfrac12 \norm{\dot{q}}^2 - \tfrac12 \real\klabig{\kla{\cos(h\Omega) \Phi q}^* g(\Phi q)} - \tfrac18 h^2 \norm{\Psi_1 g(\Phi q)}^2.
\]

Up to now, all calculations hold for $g=-\nabla U$ stemming from general potentials $U$. Now, we use that $U$ is a quadratic potential~\eqref{eq-A}, i.e., $g(q)=-A q$ with a self-adjoint matrix $A$. This implies
\[
\overline{\kla{\Phi q_{n}}^* g(\Phi q_{n+1})} = \kla{\Phi q_{n+1}}^* g(\Phi q_{n}),
\]
and hence we get from~\eqref{eq-ec-nonlinear} exact conservation of $\mathcal{H}$ in this case.
\end{proof}

We note that the numerical solution~\eqref{eq-trigo} after one time step of length $h$ does not necessarily coincide with the solution at time $h$ of the Hamiltonian differential equation corresponding to the modified energy of Theorem~\ref{thm-ec} (the considered methods are not symplectic). In other words, the modified energy of Theorem~\ref{thm-ec} is not a modified energy in the sense of a backward error analysis.

\begin{remark}
  In the special case $\Omega=0\in\mathbb{R}^{d\times d}$, the trigonometric intergator~\eqref{eq-trigo} with filters $\phi(\xi)=1$ and $\psi_1(\xi)=\sinc(\xi)$ reduces to the well-known St\"ormer--Verlet method (or leapfrog method) applied to~\eqref{eq-ode} (with $\Omega=0$). The modified energy of Theorem~\ref{thm-ec} then simplifies to
  \[
\mathcal{H}(q_n,\dot{q}_n) = \tfrac12 \norm{\dot{q}_n}^2 + \tfrac12 q_n^* A q_n - \tfrac18 h^2 \norm{A q_n}^2.
  \]
  After replacing $\dot{q}_n$ with $\frac{q_{n+1}-q_n}{h} + \frac12 h A q_n$ (which follows from~\eqref{eq-trigo-1}), we thus get from Theorem~\ref{thm-ec} that
  \[
\tfrac12 \normbig{\tfrac1h (q_{n+1}-q_n)}^2 + \tfrac12 \real\klabig{q_{n+1}^* A q_n} 
  \]
  is preserved in time. This quantity is the known discrete energy of the St\"ormer--Verlet method in the linear case (see, e.g., \cite[Section~7.1.3]{Cohen2017}), which is usually obtained by scalar multiplying the two-step formulation $q_{n+1}-2q_n+q_{n-1} = -h^2 A q_n$ of the St\"ormer--Verlet method with $\tfrac1{2h^2}(q_{n+1}-q_{n-1})$ and taking the real part. Theorem~\ref{thm-ec} gives an alternative (and longer) proof that is based on the formulation of the method as a splitting intergator. 
\end{remark}

As we show next, the modified energy~\eqref{eq-Ec} of Theorem~\ref{thm-ec} is close to the original energy~\eqref{eq-energy}. 

\begin{lemma}\label{lemma-Hc}
Under the assumption~\eqref{eq-filter-bounds}, 
we have for the modified energy $\mathcal{H}$ of Theorem~\ref{thm-ec}
\begin{align*}
\absbig{\mathcal{H}(q,\dot{q}) - \tfrac12 \norm{\Omega q}^2 - \tfrac12 \norm{\dot{q}}^2} &\le \klabig{\breve{C}+\widehat{C} h^2} \norm{q}^2, \\
\absbig{\mathcal{H}(q,\dot{q}) - H(q,\dot{q})} &\le \widetilde{C} \min\klabig{h,\omega^{-1}} \norm{q} \, \norm{\Omega q} + \widehat{C} h^2 \norm{q}^2
\end{align*}
with $\breve{C} = \tfrac12 c_0^2 \norm{A}$, $\widehat{C} = \tfrac18 c_0^4 \norm{A}^2$ and $\widetilde{C} = \tfrac12 \klabig{2 c_0^2 + (c_0+1)\max(c_0+1,c_1)} \norm{A}$.
\end{lemma}
\begin{proof}
The first estimate follows directly from the definition of the modified energy $\mathcal{H}$ using the Cauchy--Schwarz inequality, $\norm{\cos(h\Omega)}\le 1$ and $\norm{\Phi}\le c_0$ and $\norm{\Psi_1}\le c_0$ by~\eqref{eq-filter-bounds}. For the second estimate, we note that
\begin{multline*}
\mathcal{H}(q,\dot{q}) - H(q,\dot{q}) = \tfrac12 \real\klabig{ \klabig{\kla{\cos(h\Omega)-1}\Phi q}^* A \Phi q} + \tfrac12 \real\klabig{\klabig{\kla{\Phi-1} q}^* A \Phi q}\\ + \tfrac12 \real\klabig{q^*A\klabig{\kla{\Phi-1} q}} - \tfrac18 h^2 \norm{\Psi_1 A \Phi q}^2.
\end{multline*}
The Cauchy--Schwarz inequality and $\norm{\Phi}\le c_0$ and $\norm{\Psi_1}\le c_0$ yield
\begin{equation}\label{eq-proof-HmHc}
\absbig{ \mathcal{H}(q,\dot{q}) - H(q,\dot{q}) }\le \tfrac12 \klaBig{ c_0^2 \norm{\kla{\cos(h\Omega)-1} q} + (c_0+1) \norm{\kla{\Phi-1} q} } \norm{A} \, \norm{q} + \widehat{C} h^2 \norm{q}^2.
\end{equation}
It remains to bound $\cos(h\Omega)-1$ and $\Phi-1$. On the one hand, we have the bounds
\[
\norm{\kla{\cos(h\Omega)-1} q}\le h \norm{\Omega q} \qquad\text{and}\qquad \norm{\kla{\Phi-1}q}\le c_1 h \norm{\Omega q},
\]
which follow from $\abs{\cos(h\omega_j)-1} = 2 \abs{\sin(\tfrac12 h\omega_j)^2}\le h\omega_j$ and $\abs{\phi(h\omega_j)-1} \le c_1 h \omega_j$ by~\eqref{eq-filter-bounds}. 
On the other hand, we have the bounds 
\[
\norm{\kla{\cos(h\Omega)-1} q}\le 2 \omega^{-1} \norm{\Omega q} \qquad\text{and}\qquad \norm{\kla{\Phi-1}q}\le (c_0+1) \omega^{-1} \norm{\Omega q},
\]
which follow from $\abs{\cos(h\omega_j)-1} =\abs{\phi(h\omega_j)-1}=0= \omega^{-1} \omega_j$ for $\omega_j=0$ by~\eqref{eq-filter-bounds} and $\abs{\cos(h\omega_j)-1}\le 2 \le 2\omega^{-1} \omega_j$ and $\abs{\phi(h\omega_j)-1}\le c_0+1 \le (c_0+1)\omega^{-1} \omega_j$ for $\omega_j\ne 0$ by~\eqref{eq-filter-bounds} (recall that $\omega$ is the minimal nonzero frequency). Inserting these bounds of $\cos(h\Omega)-1$ and $\Phi-1$ into~\eqref{eq-proof-HmHc} yields the statement of the lemma.
\end{proof}

\section{Dynamical consequences}

From Theorem~\ref{thm-ec} and Lemma~\ref{lemma-Hc} we can deduce regularity of the numerical solution in the sense that $\Omega q_n$ and $\dot{q}_n$ stay bounded provided that $q_n$ stays bounded.

\begin{lemma}[Regularity of the numerical solution]\label{lemma-reg}
In the situation of Theorem~\ref{thm-ec} and under the assumption~\eqref{eq-filter-bounds}, we have 
\[
\norm{\Omega q_n} + \norm{\dot{q}_n} \le C \qquad\text{for all }\, n\in\mathbb{N}
\]
with $C$ depending only on $c_0$, $\norm{A}$, $\norm{q_0}$, $\norm{\Omega q_0}$, $\norm{\dot{q}_0}$ and $\norm{q_n}$.
\end{lemma}
\begin{proof}
By combining Theorem~\ref{thm-ec} and the first estimate of Lemma~\ref{lemma-Hc}, we get
\begin{align*}
  \tfrac12 \norm{\Omega q_n}^2 + \tfrac12 \norm{\dot{q}_n}^2 &\le \abs{\mathcal{H}(q_n,\dot{q}_n)} + \klabig{\breve{C}+\widehat{C}h^2} \norm{q_n}^2 \\
  &= \abs{\mathcal{H}(q_0,\dot{q}_0)} + \klabig{\breve{C}+\widehat{C}h^2} \norm{q_n}^2 \\
 &\le \tfrac12 \norm{\Omega q_0}^2 + \tfrac12 \norm{\dot{q}_0}^2 + \klabig{\breve{C}+\widehat{C}h^2}\norm{q_0}^2 + \klabig{\breve{C}+\widehat{C}h^2}\norm{q_n}^2
\end{align*}
with $\breve{C}$ and $\widehat{C}$ from Lemma~\ref{lemma-Hc}.
\end{proof}

We finally prove the following result on near-conservation of energy.

\begin{theorem}[Numerical energy conservation]\label{thm-nec}
In the situation of Theorem~\ref{thm-ec} and under the assumption~\eqref{eq-filter-bounds}, we have 
\[
\absbig{H(q_n,\dot{q}_n) - H(q_0,\dot{q}_0)} \le C \min\klabig{h,\omega^{-1}} + C h^2 \qquad\text{for all }\, n\in\mathbb{N}
\]
with $C$ depending only on $c_0$, $c_1$, $\norm{A}$, $\norm{q_0}$, $\norm{\Omega q_0}$, $\norm{\dot{q}_0}$ and $\norm{q_n}$.
\end{theorem}
\begin{proof}
By combining Theorem~\ref{thm-ec} and the second estimate of Lemma~\ref{lemma-Hc}, we get
\begin{align*}
\absbig{H(q_n,\dot{q}_n) - H(q_0,\dot{q}_0)} &= \absbig{H(q_n,\dot{q}_n) - \mathcal{H}(q_n,\dot{q}_n) + \mathcal{H}(q_0,\dot{q}_0)- H(q_0,\dot{q}_0)}\\
 &\le \widetilde{C} \min\klabig{h,\omega^{-1}} \norm{q_n}\,\norm{\Omega q_n}+ \widehat{C}h^2\norm{q_n}^2\\ &\qquad + \widetilde{C} \min\klabig{h,\omega^{-1}} \norm{q_0}\,\norm{\Omega q_0} + \widehat{C}h^2\norm{q_0}^2
\end{align*}
with $\widetilde{C}$ and $\widehat{C}$ from Lemma~\ref{lemma-Hc}.
The dependence on $\norm{\Omega q_n}$ is no problem since Lemma~\ref{lemma-reg} provides a bound for this term in dependence of $c_0$, $\norm{A}$, $\norm{q_0}$, $\norm{\Omega q_0}$, $\norm{\dot{q}_0}$ and $\norm{q_n}$.
\end{proof}

Theorem~\ref{thm-nec} gives a conceptually different proof of the energy conservation properties of \cite[Theorems 3.1 and 3.2]{Hairer2000}. It (slightly) extends these results in that an arbitrary but finite number of different high frequencies is considered instead of just a single high frequency. On the other hand, it should be mentioned that the proof of~\cite{Hairer2000} extends to the oscillatory component of the total energy and paved to way to the technique of modulated Fourier expansions, which is used since then to gain insight into the nonlinear situation.

\section{Refined estimates}

The results of the previous section do not require bounds on $\Omega q_n$ (the numerical solution multiplied with the high frequencies), but only on $q_n$. In this section, we show how to get also rid of this weaker regularity assumption provided that there are no zero frequencies, i.e., if
\begin{equation}\label{eq-omnonzero}
\omega_j \ne 0 \qquad\text{for all }\, j=1,\dots,d.
\end{equation}

This assumption of nonzero frequencies allows us to control $\norm{q}$ by $\norm{\Omega q}$, which leads to the following bound of the numerical solution~$q_n$.

\begin{lemma}\label{lemma-reg-weak}
  Let $\omega\ge \tfrac12 c_0^2\norm{A}+1$.
In the situation of Theorem~\ref{thm-ec} and under the assumptions~\eqref{eq-filter-bounds} and~\eqref{eq-omnonzero}, 
we have 
\[
\norm{q_n} \le C
\]
with $C$ depending only on $c_0$, $\norm{A}$, $\norm{q_0}$, $\norm{\Omega q_0}$ and $\norm{\dot q_0}$. 
\end{lemma}
\begin{proof}
  From Lemma~\ref{lemma-Hc} and $h\le 1$, we get
  \[
\tfrac12 \norm{\Omega q}^2 + \tfrac12 \norm{\dot q}^2 \le \abs{\mathcal{H}(q,\dot{q})} + \klabig{\breve{C}+\widehat{C}} \norm{q}^2 
\]
with the constants $\breve{C}$ and $\widehat{C}$ of Lemma~\ref{lemma-Hc}.
We then use that $\norm{\Omega q}^2 \ge \omega^2 \norm{q}^2$ by assumption~\eqref{eq-omnonzero} and that $\kla{\breve{C}+\widehat{C}} \norm{q}^2 \le \tfrac12 (\omega^2-1)\norm{q}^2$ by the stated condition on $\omega$ (which implies $\omega^2-1\ge c_0^2\norm{A}+\tfrac14 c_0^4\norm{A}^2 = 2\breve{C}+2\widehat{C}$). This yields
\[
\tfrac12 \norm{q}^2 + \tfrac12 \norm{\dot q}^2 \le \abs{\mathcal{H}(q,\dot{q})}.
\]
For $q=q_n$, we thus have
  \[
\tfrac12 \norm{q_n}^2 + \tfrac12 \norm{\dot q_n}^2 \le \abs{\mathcal{H}(q_n,\dot{q_n})} = \abs{\mathcal{H}(q_0,\dot{q_0})}
\]
by Theorem~\ref{thm-ec}. The statement of the lemma then follows by using the first estimate of Lemma~\ref{lemma-Hc} to estimate $\mathcal{H}(q_0,\dot{q_0})$ in terms of $c_0$, $\norm{A}$, $\norm{q_0}$, $\norm{\Omega q_0}$ and $\norm{\dot q_0}$. 
\end{proof}

This result yields the following unconditional version of Theorem~\ref{thm-nec} on numerical energy conservation. (By unconditional, we mean that the constants do not depend in any way on the numerical solution.) 

\begin{theorem}[Unconditional numerical energy conservation]\label{thm-nec-strong}
  Let $\omega\ge \tfrac12 c_0^2\norm{A}+1$.
In the situation of Theorem~\ref{thm-ec} and under the assumptions~\eqref{eq-filter-bounds} and~\eqref{eq-omnonzero}, we have 
\[
\absbig{H(q_n,\dot{q}_n) - H(q_0,\dot{q}_0)} \le C \min\klabig{h,\omega^{-1}} + C h^2 \qquad\text{for all }\, n\in\mathbb{N}
\]
with $C$ depending only on $c_0$, $c_1$, $\norm{A}$, $\norm{\Omega q_0}$ and $\norm{\dot{q}_0}$.
\end{theorem}
\begin{proof}
Theorem~\ref{thm-nec} yields the same statement, but with a constant that depends in addition on $\norm{q_n}$ and $\norm{q_0}$. To get rid of these additional dependencies we use Lemma~\ref{lemma-reg-weak} and $\norm{q_0} \le \omega^{-1} \norm{\Omega q_0}$ by~\eqref{eq-omnonzero}.
\end{proof}

\section{Application to linear wave equations}

In this final section, we apply Theorem~\ref{thm-nec-strong} on numerical energy conservation to linear wave equations. We consider the Klein--Gordon equation with a real-valued potential $V=V(x)$ in one space dimension:
\begin{equation}\label{eq-wave}
\partial_t^2 u = \partial_x^2 u - \rho u - V u, \qquad u=u(x,t)
\end{equation}
with a nonnegative real parameter~$\rho$. 
In space, we impose $2\pi$-periodic boundary conditions. Therefore, we also assume that the potential~$V(x)=\sum_{j=-\infty}^{\infty} V_j \e^{\iu j x}$ is $2\pi$-periodic and that it belongs to the Sobolev space $H^1(\mathbb{T})$. 

For a Fourier collocation in space, we replace $u=u(x,t)$ by a (spatial) trigonometric polynomial of degree~$K$:
\[
u^K=u^K(x,t) = \sum_{j=-K}^{K-1} q_j(t) \e^{\iu j x} .
\]
Requiring that this ansatz satisfies the wave equation~\eqref{eq-wave} in the collocation points $x_k=k\pi/K$, $k=-K,\dots,K-1$, yields
\[
\partial_t^2 u^K = \partial_x^2 u^K - \rho u^K - \mathcal{I}^K \klabig{V u^K} 
\]
with initial values
\[
u^K(\cdot,0) = \mathcal{I}^K\klabig{u(\cdot,0)}, \qquad \partial_t u^K(\cdot,0) = \mathcal{I}^K\klabig{\partial_t u(\cdot,0)}
\]
and with the trigonometric interpolation~$\mathcal{I}^K$ by a trigonometric polynomial of degree~$K$. Written in Fourier space, this becomes the ordinary differential equation
\begin{equation}\label{eq-wave-semi}
\ddot{q} = - \Omega^2 q - A q, \qquad q=q(t) = \klabig{q_j(t)}_{j=-K}^{K-1}
\end{equation}
with the matrices
\[
\Omega = \diag(\omega_j)_{j=-K}^{K-1}, \qquad \omega_j = \sqrt{j^2+\rho}
\]
and
\[
A = (a_{j l})_{j,l=-K}^{K-1}, \qquad a_{j l} = \sum_{m=-\infty}^{\infty} V_{j-l+2Km} .
\]
(Note that $\mathcal{I}^K(V u^K) = \sum_{j=-K}^{K-1} (A q)_j \e^{\iu j x}$.)
The matrix~$A$ is self-adjoint since~$V$ is real-valued, and hence $V_j=\overline{V_{-j}}$ for all~$j$. After reindexing (indices from $1$ to $2K+1$ instead of $-K$ to $K-1$), the equation~\eqref{eq-wave-semi} is thus of the form~\eqref{eq-ode} with~\eqref{eq-A}. The energy~$H$ is again given by~\eqref{eq-energy}. 

In order to apply Theorem~\ref{thm-nec-strong} on numerical energy conservation for~\eqref{eq-wave-semi} we need to control $\norm{A}$.

\begin{lemma}\label{lemma-A-wave}
  We have
  \[
\norm{A} \le c_2 \norm{V}_{H^1(\mathbb{T})}
\]
with an absolute constant $c_2$ independent of $K$ and $V$.
\end{lemma}
\begin{proof}
  We show that
  \begin{equation}\label{eq-proof-A-wave}
  \norm{A q} \le C \norm{V}_{H^1(\mathbb{T})} \norm{q}
  \end{equation}
  for all vectors~$q$, which yields the statement of the lemma.
  Note that, by the Parseval equality,
  \[
  \norm{A q} = \normbig{\mathcal{I}^K \klabig{ V u^K }}_{L^2(\mathbb{T})} = \normbig{\mathcal{I}^K \klabig{ \mathcal{I}^K(V) u^K }}_{L^2(\mathbb{T})}
  \]
  with $u^K(x)=\sum_{j=-K}^{K-1} q_j \e^{\iu j x}$. We then use the product estimate $\norm{\mathcal{I}^K(v^K u^K)}_{L^2(\mathbb{T})} \le C \norm{v^K}_{H^1(\mathbb{T})} \norm{u^K}_{L^2(\mathbb{T})}$ for trigonometric polynomials $v^K$ and $u^K$ (see, e.g., Proposition~3.1~(i) of~\cite{Gauckler2015} with $\sigma=0$ and $\sigma'=1$). This yields
  \[
\norm{A q} \le C \normbig{\mathcal{I}^K(V)}_{H^1(\mathbb{T})} \normbig{u^K}_{L^2(\mathbb{T})} = C \normbig{\mathcal{I}^K(V)}_{H^1(\mathbb{T})} \norm{q}.
\]
The estimate~\eqref{eq-proof-A-wave} finally follows from $\norm{\mathcal{I}^K(V)}_{H^1(\mathbb{T})}\le C \norm{V}_{H^1(\mathbb{T})}$ (see, e.g., Lemma~4.2 of~\cite{Hairer2008} with $s=1$).
\end{proof}

We finally get the following result on all-time near-conservation of energy for the linear wave equation (where we write $q_n=(q_{n,j})_{j=-K}^{K-1}$ for the numerical solution of~\eqref{eq-trigo}).

\begin{theorem}[Numerical energy conservation for the linear wave equation]\label{thm-wave}
  Let the parameter~$\rho$ in the linear wave equation~\eqref{eq-wave} satisfy $\rho\ge \tfrac12 c_0^2 c_2 \norm{V}_{H^1(\mathbb{T})} +1$ with the constant $c_2$ of Lemma~\ref{lemma-A-wave}. Consider the spatial semi-discretization~\eqref{eq-wave-semi} of the linear wave equation~\eqref{eq-wave} that is discretized in time with a trigonometric integrator~\eqref{eq-trigo} whose filter functions satisfy~\eqref{eq-filter-bounds} and~\eqref{eq-hl}. Then, we have
\[
\absbig{H(q_n,\dot{q}_n) - H(q_0,\dot{q}_0)} \le C h \qquad\text{for all }\, n\in\mathbb{N}
\]
with $C$ depending only on $c_0$, $c_1$, $c_2$, $\norm{\Omega q_0}=\norm{u^K(\cdot,0)}_{H^1(\mathbb{T})}$, $\norm{\dot{q}_0}=\norm{\partial_t u^K(\cdot,0)}_{L^2(\mathbb{T})}$ and $\norm{V}_{H^1(\mathbb{T})}$. 
\end{theorem}
\begin{proof}
This follows from Theorem~\ref{thm-nec-strong} and Lemma~\ref{lemma-A-wave} since the frequencies $\omega_j=\sqrt{j^2+\rho}$ satisfy~\eqref{eq-omnonzero} and since $\omega=\rho\ge 1$.
\end{proof}

\begin{remark}
  (a) If the condition on~$\rho$ of Theorem~\ref{thm-wave} does not hold, then we can (only) apply Theorem~\ref{thm-nec} which yields near-conservation of energy (only) as long as $\norm{q_n}$ stays bounded.

  (b) Theorem can be extended (with the same proof) to wave equations~\eqref{eq-wave} in higher spatial dimensions~$m$, provided that the potential~$V$ belongs to the Sobolev space $H^s(\mathbb{T}^m)$ for some $s>\tfrac{m}2$.
\end{remark}

Theorem~\ref{thm-wave} states near-conservation of energy on long time intervals (even for all times!) by certain symmetric methods when applied to a Hamiltonian partial differential equation, without requiring any regularity of the numerical solution. While previous results in this direction for linear situations (see \cite{Dujardin2007,Castella2009,Debussche2009,Faou2012}) and nonlinear situations (see \cite{Cohen2008,Gauckler2010b,Faou2009a,Faou2011,Faou2012,Gauckler2017a}, for example) require CFL-type restrictions on the discretization parameters or need to exclude resonant or near-resonant time step-sizes, Theorem~\ref{thm-wave} is completely uniform in both discretization parameters.

\subsection*{Acknowledgement}

This work was supported by Deutsche Forschungsgemeinschaft through SFB 1114.

\end{document}